\documentclass{amsart}

\usepackage[english]{babel}
\usepackage{amssymb}
\usepackage{amsmath}
\usepackage{txfonts}

\newcommand{\bv}{{\mathbf{v}}}

\newcommand{\Ps}{{\mathbf{P}}}
\newcommand{\Z}{{\mathbf{Z}}}
\newcommand{\C}{{\mathbf{C}}}

\newcommand{\Q}{{\mathbf{Q}}}

\newcommand{\fr}[1]{\{#1\}}
\newcommand{\frd}[1]{\left\{#1\right\}}

\newcommand{\bk}{\mathbf{k}}

\newcommand{\ve}{\mathbf{e}}

\renewcommand{\phi}{\varphi}
    \newtheorem{lemma}{Lemma}[section]
    \newtheorem{proposition}[lemma]{Proposition}
    \newtheorem{theorem}[lemma]{Theorem}
    
    \newtheorem{corollary}[lemma]{Corollary}

   \theoremstyle{definition}
    \newtheorem{definition}[lemma]{Definition}
    
    \newtheorem{example}[lemma]{Example}
    \newtheorem{remark}[lemma]{Remark}
    
    \DeclareMathOperator{\rank}{rank}
    \DeclareMathOperator{\Pic}{{Pic}}

\DeclareMathOperator{\prim}{{prim}}

\DeclareMathOperator{\spa}{{span}}

\DeclareMathOperator{\ord}{{ord}}

\usepackage[all]{xy}
\begin{document}
\title{A structure theorem for  fibrations on Delsarte surfaces}
\author{Bas Heijne}
\address{Instut f\"ur algebraische Geometrie, Leibniz Universit\"at Hannover, Welfengarten 1, D-30167 Hannover, Germany}
\email{heijne@math.uni-hannover.de}
\author{Remke Kloosterman}
\address{Institut f\"ur Mathematik, Humboldt-Universit\"at zu Berlin, Unter den Linden 6, D-10099 Berlin, Germany}
\email{klooster@math.hu-berlin.de}
\thanks{This paper is inspired by the results of  \cite[Section 6]{HeijnePhD} and by some discussions which took place on the occasion of the PhD defense of the first author at the University of Groningen. 
Part of the research was done while the first author held a position at the University of Groningen. His position was supported by a grant of the Netherlands Organization for Scientific Research (NWO).
The research is partly supported by ERC Starting grant 279723 (SURFARI).
The second author acknowledges the hospitality of the University of Groningen and the Leibniz Universit\"at Hannover, where most of the work was done.}

\begin{abstract}In this paper we study a special class of fibrations on Delsarte surfaces. 
We call these fibrations Delsarte fibrations. 
We show that after a specific cyclic base change 
 the fibration is the pull back of a fibration with three singular fibers, and that this second base change is completely ramified at  two points where the fiber is singular.

As a corollary we show that every Delsarte fibration of genus 1 with nonconstant $j$-invariant occurs as the base change of an elliptic surface from Fastenberg's list of rational elliptic surfaces with $\gamma<1$.
\end{abstract}
\date{\today}
\maketitle
\section{Introduction}
 A Delsarte surface $S$ is a surface of $\Ps^3$ defined by the vanishing  of a polynomial $F$ consisting of four monomials. Let $A$ be the exponent matrix of $F$, then a Delsarte surface is the quotient of a Fermat surface if and only if $\det(A)\neq 0$. Shioda used this observation in \cite{ShiodaPic} to present an algorithm to determine the Lefschetz number of any smooth surface that is birationally equivalent with $S$.

Fix now two disjoint lines $\ell_1,\ell_2$ in $\Ps^3$. The projection with center $\ell_1$ onto $\ell_2$ yields a rational map $S\dashrightarrow \Ps^1$. Resolving the indeterminancies of this map yields a fibration $\tilde{S} \to \Ps^1$.
If the genus of the general fiber is one and this morphism has a section then Shioda's algorithm together with the Shioda-Tate formula allows one to determine the Mordell-Weil rank of the group of sections. Shioda applied this to the surface
\[ y^2+x^3+1+t^n\]
and showed in \cite{ShiodaAstr} that the maximal Mordell-Weil rank (by varying $n$) is 68.

If both lines $\ell_i$ are intersections of two coordinate hyperplane then we call the obtained fibration a \emph{Delsarte fibration}. We will introduce the notion of a \emph{Delsarte base change}. Roughly said, this is a base change $\Ps^1\to\Ps^1$ completely ramified over $0$ and $\infty$. In particular, the pullback of a Delsarte fibration under a Delsarte base change is again a  Delsarte fibration.
The first author determined in his PhD thesis \cite{HeijnePhD} the maximal Mordell-Weil rank under Delsarte base changes of any Delsarte fibration such that the general fiber has genus one. In this way he showed that Shioda's example has the highest possible rank among Delsarte fibration of genus one.

In \cite{FastExtremal} and \cite{FastLongTable}  Fastenberg calculated the maximal Mordell-Weil rank under base changes $t\mapsto t^n$ for a special class of elliptic surface, i.e., elliptic curves over $\C(t)$ with nonconstant $j$-invariant such that  a certain invariant $\gamma$ is smaller than $1$. It turned out that all the ranks that occur for Delsarte surfaces with nonconstant $j$-invariant also occur in Fastenberg's list. In \cite[Ch. 6]{HeijnePhD} it is shown that for every Delsarte fibration of genus one, there exist integers $m,n$ such that the Delsarte base change of degree $m$ of the Delsarte fibration is isomorphic to a base change of the form $t\mapsto t^n$ of one of the surfaces in Fastenberg's list.

In this paper we present a more conceptual proof for this phenomenon: First we study the configuration of singular fibers of a Delsarte fibration. We show that for any Delsarte fibration each two singular fibers over points $t\neq0,\infty$ are isomorphic. Then we show that after a base change of the form $t\mapsto t^n$ the Delsarte fibration is a base change of the form $t\mapsto t^m$ of a fibration with at most one singular fiber away from $0,\infty$.

If there is no singular fiber away from $0,\infty$, then the fibration becomes split after a base change of $t\mapsto t^m$. If there is at least one singular fiber then we show that there are three possibilities, namely the function field extension $K(S)/K(\Ps^1)=K(x,y,t)/K(t)$  is given by $m_1+m_2+(1+t)m_3$, where the $m_i$ are monomials in $x$ and $y$, or this extension is given by $y^a=x^b+x^c+tx^d$, where $b,c,d$ are mutually distinct, or the singular fiber away from $0$ and $\infty$ has only nodes as singularities and is therefore semistable. See Proposition~\ref{propStandardForm}.

In the case of a genus one fibration we can use this classification to check almost immediately that any Delsarte fibration of genus one admits a base change of the form $t\mapsto t^n$ such that the pulled back fibration is the pull back of a fibration with $\gamma<1$ or has a constant $j$-invariant. See Corollary~\ref{corGamma}. This procedure is carried out in Section~\ref{secDes}. 

 The techniques used in the papers by Fastenberg use the fact that the fibration is not isotrivial and it seems very hard to extend these techniques to isotrivial fibrations.
In Section~\ref{secIsoTrivial} we consider an example of a class of isotrivial Delsarte fibrations. Shioda's algorithm yields the Lefschetz number of any Delsarte surface with $\det(A)\neq 0$. Hence it is interesting to see how it works in the case where Fastenberg's method breaks down. Let $p$ be an odd prime number, and $a$ a positive integer. We consider the family of surfaces
\[ S:y^2=x^p+t^{2ap}+s^{2ap}\]
in $\Ps(2a,ap,1,1)$. Then $S$ is birational to a Delsarte surface. After blowing up $(1:1:0:0)$ we obtain a smooth surface $\tilde{S}$ together with a morphism $\tilde{S}\to \Ps^1$. The general fiber of this morphism is a hyperelliptc curve of genus $(p-1)/2$. We show that if $p>7$ then $\rho(\tilde{S})=2+6(p-1)$, in particular the Picard number  is independent of $a$. Two of the generators of the N\'eron-Severi group of $S$ can be easily explained: the first one is the pullback of the hyperplane class on $S$, the second class is the exceptional divisor of the morphism $\tilde{S}\to S$. In Example~\ref{exaIso} we give also equations for some other classes.

If we take $p=3$ then we find back Shioda's original example. However, in Shioda's example one has that $\rho(\tilde{S})$ is not completely independent of $a$, it depends namely on $\gcd(a,60)$. Similarly one can show that if $p=5$ and $p=7$  then $\rho(\tilde{S})$ depends on $a$. Our result shows that these cases are exceptions, i.e., for $p>7$ we have that $\rho(\tilde{S})$ is completely independent of $a$.

\section{Delsarte surfaces}\label{secDes}
In this section we work over an algebraically closed field $K$ of characteristic zero.

\begin{definition} \label{defBasic} A surface $S\subset \Ps^3$ is called a \emph{Delsarte surface} if $S$ is the zero-set of a polynomial of the form
\[F:=\sum_{i=0}^3 c_i \prod_{j=0}^3 X_i^{a_{i,j}},\]
with $c_i\in K^*$ and $a_{i,j}\in \Z_{\geq0}$. The $4\times 4$ matrix $A:=(a_{i,j})$ is called the \emph{exponent matrix} of $S$. 

A \emph{Delsarte fibration of genus $g$}  on a Delsarte surface $S$ consists of the choice of two disjoint lines $\ell_1,\ell_2$ such that both the $\ell_i$ are the intersection of two coordinate hyperplanes and the generic fiber of the projection $S\dashrightarrow \ell_2$ with center $\ell_1$ is an irreducible curve of geometric genus $g$.

A \emph{Delsarte birational map} is a birational map $\varphi:\Ps^3\dashrightarrow \Ps^3$ such that $\varphi(X_0:\dots:X_3)= (\prod X_j^{b_{0j}}:\dots: \prod X_j^{b_{3j}})$, i.e., $\varphi$ is a birational monomial map.
\end{definition}

\begin{remark} \label{rmkCoeff} Since $K$ is algebraically closed, we can multiply each of the four coordinates  $X_i$ by a nonzero constant such that all four constants in $F$ coincide, hence without loss of generality we may assume that $c_i=1$.

After permuting the coordinates, if necessary, we may assume that $\ell_1$ equals $V(X_2,X_3)$ and $\ell_2$ equals $V(X_0,X_1)$. Then the projection map $S\dashrightarrow \ell_2$ is just the map $[X_0:X_1:X_2:X_3]\to [X_2:X_3]$. Let $f(x,y,t):=F(x,y,t,1)\in K(t)[x,y]$.
Then the function field extension $K(S)/K(\ell_2)$ is isomorphic to the function field extension $K(x,y,t)/f$ over $K(t)$.

We call a Delsarte fibration with $\ell_1=V(X_2,X_3)$ and $\ell_2=V(X_0,X_1)$ the \emph{standard fibration} on $S$.
\end{remark}

\begin{definition}\label{defBaseChange}
Let $n$ be a nonzero integer. A \emph{Delsarte base change of degree $|n|$} of a Delsarte fibration $\varphi :S\dashrightarrow \Ps^1$ is a Delsarte surface $S_n$, together with a Delsarte fibration $\varphi_n: S_n\dashrightarrow \Ps^1$ and a Delsarte rational map $S_n\dashrightarrow S$ of degree $n$, such that there exists a commutative diagram
\[
\xymatrix{S_n\ar@{-->}[r] \ar@{-->}[d] &S\ar@{-->}[d] \\\
 \Ps^1 \ar[r]&\Ps^1}
\]
and $K(S_n)/K(\Ps^1)$ is isomorphic to the function field extension $K(x,y,s)/(f(x,y,s^n)$ over $K(s)$.
\end{definition}

\begin{remark}
Note that $n$ is allowed to be negative. If $n$ is negative then a base change of degree $-n$ is the composition of the automorphism $t\mapsto 1/t$ of $\Ps^1$ with the usual degree $-n$ base change $t\mapsto t^{-n}$. In many cases we  compose a base change with a Delsarte birational map which respects the standard fibration. In affine coordinates such a map is given by $(x,y,s)\mapsto (xs^a,ys^b,s^n)$ for some integers $a,b$.
\end{remark}

\begin{lemma}\label{lemRatFib} Let $S$ be a Delsarte surface. Suppose there is a nonzero vector $\bv=(a,b,0,0)^T$ in $\Z^4$ such that $A\bv\in \spa(1,1,1,1)^T$. Then the generic fiber of the standard fibration $\varphi:S\to \Ps^1$ is a rational curve. 
\end{lemma}
\begin{proof}
 After interchanging $x$ and $y$, if necessary, we may assume that $a$ is nonzero. Consider now $f_0:=f(x^a,x^by,t)$. The exponents of $x$ in the four monomials of $f_0$ are precisely the entries of $A\bv$. Since $A\bv=e(1,1,1,1)^T$ for some integer $e$ we have that $f_0=x^eg(y,t)$. This implies that the generic fiber of $\varphi$ is dominated by a finite union of rational curves. Since the generic fiber is irreducible it follows that the generic fiber of $\varphi$ is a rational curve.
\end{proof}

\begin{lemma} \label{lemProd} Let $S$ be a Delsarte surface. Suppose there is a nonzero vector $\bv=(a,b,c,0)^T$ in $\Z^4$ such that $c\neq 0$ and $A\bv\in \spa(1,1,1,1)^T$. Then there is a Delsarte base change of degree $|c|$ such that the pull back of the standard fibration on $S$ is birational to a product $C\times \Ps^1\to\Ps^1$.
\end{lemma}
\begin{proof}
Consider now $f_0:=f(xt^a,yt^b,t^c)$. The exponents of $t$ in the four monomials of $f_0$ are precisely the entries of $A\bv$.
Since $A\bv=e(1,1,1,1)^T$ for some integer $e$ we have that $f_0=t^eg(x,y)$. Let $S'$ be the projective closure of $g=0$ in $\Ps^3$. Then $S'$ is a cone over the plane curve $g=0$, in particular $S'$ is birational to $C\times \Ps^1$ and the standard fibration on $S'$ is birational to the projection $C\times \Ps^1\to \Ps^1$.
Now $S'$ is birational to the surface $S_c$, the projective closure of $f(x,y,t^c)=0$. Hence $S_c \to \Ps^1$ is birational to $C\times \Ps^1\to \Ps^1$.
\end{proof}

\begin{lemma}\label{lemDetA} Let $A$ be the exponent matrix of a Delsarte surface. There exists a nonzero vector $\bv=(a,b,c,0)^T$ in $\Z^4$ such that $A\bv\in \spa(1,1,1,1)^T$ if and only if $\det(A)=0$.
\end{lemma}

\begin{proof}
Since each row sum of $A$ equals $d$, the degree of the $i$-th monomial in $F$, it follows that  $A(1,1,1,1)^T =d(1,1,1,1)^T$. Suppose first that $\det(A)\neq 0$.  Then from $A(1,1,1,1)^T =d(1,1,1,1)^T$  it follows that $A^{-1}(1,1,1,1)^T\in\spa (1,1,1,1)^T$, which does not contain a nonzero vector with vanishing fourth coordinate.

Suppose now that $\det(A)=0$. Denote with $A_i$ the $i$-th column of $A$. From the fact that each row sum of $A$ equals $d$ we get that $A_1+A_2+A_3+A_4=d(1,1,1,1)^T$.
Since $\det(A)=0$ there exists a nonzero vector $(a_1,a_2,a_3,a_4)$ such that $\sum a_iA_i=0$.  From this we obtain
\[ (a_4-a_1)A_1+(a_4-a_2)A_2+(a_4-a_3)A_3=a_4(A_1+A_2+A_3+A_4)= a_4d(1,1,1,1)^T.\]
I.e., $\bv=(a_4-a_1,a_4-a_2,a_4-a_3,0)^T$ is a vector such that $A\bv\in \spa (1,1,1,1)^T$. We need to show that $\bv$ is nonzero. Suppose the contrary, then also $A\bv=a_4d(1,1,1,1)^T$ is zero and therefore $a_4=0$. Substituting this in $\bv$ yields that  $\bv=(-a_1,-a_2,-a_3,0)=(0,0,0,0)$ holds, which contradicts our assumption that $(a_1,a_2,a_3,a_4)$ is nonzero.
\end{proof}

\begin{remark} \label{rmkDet} 
We want to continue to investigate the singular fibers of a Delsarte fibration, in particular the singular fibers over points $t=t_0$ with $t_0\neq 0,\infty$. 
If $\det(A)=0$ then either the generic fiber has geometric genus 0 or after a Delsarte base change the fibration is split, i.e., the fibration is birational to a product. In the latter case all the fibers away from $0$ and $\infty$ are smooth. Hence from now on we restrict to the case where $\det(A)\neq 0$.
\end{remark}

\begin{lemma} \label{lemOneT} Let $S$ be a Delsarte surface with $\det(A)\neq0$, such that the generic fiber has positive geometric genus.  Let $\varphi: S \to \Ps^1$ be the standard Delsarte fibration. Then there exists a Delsarte base change of $\varphi$ that is birational to the standard fibration on a Delsarte surface $S'$ with affine equation of the form $m_1+m_2+m_3+t^nm_4$, where each $m_i$ is a monomial in $x$ and $y$.
\end{lemma}

\begin{proof}
Let $\ve_0=(1,1,1,1)^T$ and $\ve_i$ be the $i$-th standard basis vector of $\Q^4$.
Let $V_i$ be the vector space spanned by $\ve_0$ and $\ve_i$. Since $A^{-1}\ve_0=\frac{1}{d}\ve_0$ it follows that  $A^{-1}V_i$ is not contained in $\spa \{\ve_1,\ve_2,\ve_3\}$.
In particular, $\dim A^{-1}V_i\cap \spa \{\ve_1,\ve_2,\ve_3\}=1$. 

Let $\ell_i$ be the line $A^{-1}V_i\cap \spa \{\ve_1,\ve_2,\ve_3\}$ and let $\bv_i$ be a vector spanning $\ell_i$.
We can scale $\bv_i$ such that $A\bv_i=\ve_i+t_i\ve_0$  for some $t_i \in K$. Since $\ve_0,\ve_1,\ve_2,\ve_3$ are linearly independent it follows that $\{\ve_i+t_i\ve_0\}_{i=1}^3$ are linearly independent and therefore $\bv_1,\bv_2,\bv_3$ are linearly independent. Hence $\spa \{\bv_1,\bv_2,\bv_3\}$ is three-dimensional and there is at least one $\bv_i=(a_i,b_i,c_i,0)$  with $c_i\neq0$. Then the rational map defined by $(x,y,t)\mapsto (xt^{a_i},yt^{b_t},t^{c_i})$ is a composition of a Delsarte base change and a Delsarte rational map.
 
Now three of the four entries of $A\bv_i$ coincide, say they equal $e$. The exponent of $t$ in  the four monomials of  $f_0:=f(xt^{a_i},yt^{b_i},t^{c_i})$ are the entries of $A\bv_i$. In particular, in precisely three of the four monomials the exponents of $t$ equal the same constant $e$. Therefore $g:=f_0/t^e$ consists of four monomials of which precisely one contains a $t$. If the exponent of $t$ in this monomial is negative then we replace $t$ by $1/t$ in $g$. Then $g=0$ is an affine polynomial equation for the surface $S'$.
\end{proof}

Recall that we investigate the singular fibers of a Delsarte fibration, in particular the singular fibers over points $t=t_0$ with $t_0\neq 0,\infty$. 
If we have a Delsarte fibration and take a Delsarte base change then the type of singular fiber over $t=0,\infty$ may change, since the base change map is ramified over these points. Over points with $t\neq 0,\infty$ the base change map is unramified and therefore the type of singular fibers remains the same. 
Hence to describe the possible types of singular fibers over points with $t\neq0,\infty$ it suffices by Lemma~\ref{lemOneT} to study Delsarte surfaces such that only one monomial contains a $t$, i.e., 
 we may restrict ourselves to Delsarte surfaces with affine equation $m_1+m_2+m_3+t^nm_4$. If $n=0$ then the fibration is split and there are no singular fibers. If $n\neq 0$ then the possible types of singular fibers are already determined at $n=1$, i.e., it suffices to consider Delsarte surfaces with affine equation $m_1+m_2+m_3+tm_4$. 
 
 \begin{definition}\label{defMinFib}
 We call the standard fibration on a Delsarte surface a \emph{minimal Delsarte fibration} if the following conditions hold:
 \begin{enumerate}
 \item The affine equation for the standard fibration is of the form $m_1+m_2+m_3+tm_4$, where the $m_i$ are monomials in $x$ and $y$.
 \item The exponent matrix $A$ of the corresponding surface $S\subset \Ps^3$ satisfies $\det(A)\neq 0$.
 \end{enumerate}
 \end{definition}
 
 \begin{remark}\label{remMinFib} In the function field $K(S)=K(x,y,t)/f$ we have the relation $t=(-m_1-m_2-m_3)/m_4$. In particular $K(S)=K(x,y)$ and therefore $S$ is a rational surface.
 
 Consider now the defining polynomial for $S$, i.e., $M_1+M_2+M_3+X_2M_4$, where the $M_i$ are monomials in $X_0,X_1,X_3$, the degrees of $M_1, M_2$ and $M_3$ are the same, say $d$ and the degree of $M_4$ equals $d-1$.
 
 The Delsarte fibration is induced by the map $(X_0:X_1:X_2:X_3)\mapsto (X_2:X_3)$. 
 If $S$ contains the line $\ell_1:X_2=X_3=0$ then this rational map can be extended to a morphism on all of $S$, otherwise we blow-up the intersection of this line with $S$ and obtain a morphism $\tilde{S}\to\Ps^1$, such that each fiber is  a plane curve of degree $d$.
 
 There is a different way to obtain this family of plane curves. Define $N_i'$ as follows:
\[ N_i':=M_i(X_0,X_1,X_2,X_2) \mbox{ for } i=1,2,3 \mbox{ and } N_4'=X_2M_4(X_0,X_1,X_2,X_2)\]
 Now the four $N_i'$ have a nontrivial greatest common divisor if and only if $X_3\mid M_i$ for $i=1,2,3$. The later condition is  equivalent to the condition that the line $\ell_1$ is contained in $S$. Moreover, if the greatest common divisor is nontrivial then it equals $X_2$.
 Now set $N_i=N'_i$ if $\ell_1\not \subset S$ and set $N_i=N'_i/X_3$ if $\ell_1\subset S$.
 Then $\lambda(N_1+N_2+N_3)+\mu N_4$ is a pencil of plane curves of degree $d$ or  $d-1$ and the generic member of this pencil is precisely the generic fiber of the standard fibration on $S$.
 
 We can consider the generic member of this family as a projective curve $C$ over $K(t)$ with defining polynomial $G:=N_1+N_2+N_3+tN_4\in K(t)[X_0,X_1,X_2]$. Let $A'$ be the exponent matrix of $C$ (considered as a curve in $\Ps^2_{K(t)}$). Set
  \[ B:=\left(\begin{matrix} 1&0&0\\0&1&0\\0&0&1\\ 0&0&1\end{matrix}\right) \mbox{ if } \ell_1\not\subset S \mbox{ and }  B:=\left(\begin{matrix} 1&0&\frac{-1}{d}\\0&1&\frac{-1}{d}\\0&0&\frac{d-1}{d}\\ 0&0&\frac{d-1}{d}\end{matrix}\right) \mbox{ otherwise.}\]
 Then $A'=AB$. Since $A$ is invertible and $B$ has rank $3$ it follows that $\rank A'=3$. Moreover the first three rows of $A'$ are linearly independent, since the upper $3\times 3$-minor of $A'$ equals the upper $3\times 3$-minor of $A$ times the upper $3\times 3$ minor of $B$.  
 
In particular, there is a vector $\bk$, unique up to scalar multiplication, such that $\bk A'=0$. Since the upper three rows of $A'$ are linearly independent it follows that the fourth entry $k_4$ of $\bk$ is nonzero. We can make the vector $\bk$ unique, by requiring that $k_4>0$, $k_i\in \Z$ for $i=1,\dots, 4$ and $\gcd(k_1,k_2,k_3,k_4)=1$.
Moreover from $\rank A'=3$ and the fact that $(0,0,1,-1)B$ vanishes it follows that $\bk\in \spa \{(0,0,1,-1)A^{-1}\}$.

Since none of the rows of $A'$ is zero there are at least two nonzero entries in $\bk$. Suppose that there are precisely two nonzero entries,  say $k_i$ and $k_4$. Then $-k_i$ times the $i$-th row of $A'$ equals $k_4$ times the fourth row of $A'$. Each row sum of $A'$ equals the degree of $C$, say $d$. From this it follows that $k_id=-k_4d$ and hence that  $k_i=-1,k_4=1$. In particular, the $i$-th row and the fourth row coincide. After permuting $m_1,m_2,m_3$, if necessary, we may assume that the affine equation for the standard fibration is of the form $m_1+m_2+(1+t)m_3$.

Hence if the four monomials $m_1,m_2,m_3,m_4$  (in $x,y$) are distinct then at least three of the four entries of $\bk$ are nonzero.

Let $A'_i$ be the $i$-th row of $A'$.  Recall that each row sum  of  $A'_i$ equals $d$. Since $\sum k_iA_i$ equals zero it follows that  $0=\sum_i k_i\sum_j A'_{i,j}=\sum_i k_id$ and hence $\sum k_i=0$.  Let $p$ be a prime number dividing one of the $k_i$. Since  $\gcd(k_1,\dots,k_4)=1$ there is a $j$ such that $p\nmid k_j$. From $\sum k_i=0$ it follows that there is a $j'\neq j$ such that $p\nmid k_{j'}$. Hence each prime number $p$ does not divide at least two of the entries of $\bk$.
 \end{remark}

\begin{proposition}\label{propDiscrim} Let $S\to \Ps^1$ be a minimal Delsarte fibration. Let $A'$, $\bk$ and $N_i$ be as above. Suppose that the fiber over $t=t_0$ is singular and $t_0\neq0,\infty$ then
\[  t_0^{k_4}-\prod_{i: k_i\neq 0} k_i^{k_i}=0\]
or two of the $N_i$ coincide.
\end{proposition}

\begin{proof}
Let $G\in K(t)[X_0,X_1,X_2]$ be as above. Then $G$ defines a pencil of plane curves in $\Ps^2_K$. Assume that no two of the $N_i$ coincide. We aim at determining the singular members of the pencil defined by $G$.
Let $B_t$ be the matrix obtained from  $A'$ by multiplying the fourth row  by $t$. Let us consider the matrix $B_{t_0}$ for some  $t_0\in K^*$.
Since the upper $3\times 3$ minor of $B_{t_0}$ equals the upper $3\times 3$ minor of $A'$, and this minor is nonzero, it follows that $\rank B_{t_0}=3$. Hence the kernel of right-multiplication by $B_{t_0}$ is one-dimensional and is generated by $(k_1,k_2,k_3,\frac{k_4}{t_0})$.

Consider now the closure of the image of the rational map $M:K^3\dashrightarrow K^4$ sending $(x,y,z)$ to $(N_1,N_2,N_3,N_4)$. Let $z_1,z_2,z_3,z_4$ be the coordinates on $K^4$. Then by the definition of the vector $(k_1,k_2,k_3,k_4)$ one has that $\prod N_i^{k_i}=1$ holds, i.e., on the image of $M$ one has 
\[ \prod z_i^{k_i}=1\]
Since the greatest common divisor of the $k_i$ equals one, this defines an irreducible hypersurface $\overline{V}$ in $K^4$. Moreover, from the fact that $\rank A'$ equals 3 it follows that $M$ has finite fibers, hence the image of $M$ is three-dimensional and the closure of the image of $M$ is precisely the closure of $ \prod z_i^{k_i}=1$.

We want now to determine the values for which $t_0$ the corresponding member of the pencil of plane curves is singular. Hence we want to find $(x_0:y_0:z_0)\in \Ps^2$ and $t_0\in K^*$ such that  for $(t,X_0,X_1,X_2)=(t_0,x_0,y_0,z_0)$  the vector $(G_{X_0},G_{X_1},G_{X_2})$ is zero. In particular, the vector  $(X_0G_{X_0},X_1G_{X_1},X_2G_{X_2})$ is zero. A direct calculation shows that the latter vector equals $(N_1,N_2,N_3,tN_4)A'$, which in turn equals $(N_1,N_2,N_3,N_4)B_t$. 
Hence if $(x_0,y_0,z_0)$ is a singular point of a fiber over $t=t_0$ then  $M(x_0,y_0,z_0)$ is contained in $\ker B_{t_0}\cap \overline{V}$. 

We consider first the case where  $M(x_0,y_0,z_0)$ is nonzero and $t_0\neq 0$. Then $ \prod_{i:k_i\neq 0} z_i^{k_i}=1$ and  $(z_1,z_2,z_3,z_4)$ is a multiple of $(k_1,k_2,k_3,k_4/t_0)$. In particular,
\[ \frac{\prod_{i:k_i\neq 0} k_i^{k_i}}{t_0^{k_4}}=1\] 
holds, which finishes the case where $M(x_0,y_0,z_0)$ is nonzero.

To finish we show that if $t_0\neq 0$ and $\ker B_{t_0}\cap \overline{V}$ consists only of $(0,0,0,0)$ then the fiber over $t_0$ is smooth. Since $\ker B_{t_0}\cap \overline{V}$ consists only of $(0,0,0,0)$ each singular point of the fiber satisfies $N_1=N_2=N_3=N_4=0$. In particular at least two of the $X_i$ are zero. Without loss of generality we may assume that the point $(0:0:1)$ is singular. Consider now $G(x,y,1)$ and write this as $m_1+m_2+m_3+tm_4$.

Since all the four $N_i$ are distinct we have that $m_1+m_2+m_3+tm_4=0$ is an equisingular deformation of $m_1+m_2+m_3+t_0m_4$ for $t$ in a small neighborhood of $t_0$. Hence we can resolve this singularity simultaneously for all $t$ in a neighborhood of $t_0$. Therefore all fibers in a neighborhood of $t_0$ are smooth and, in particular, the fiber over $t_0$ is smooth.
\end{proof}

\begin{lemma}\label{lemBaseAuto} Let $\varphi: S\to \Ps^1$ be a minimal Delsarte fibration. Then there is an automorphism $\sigma:S\to S$, mapping  fibers of $\varphi$ to fibers, such that its action on the base curve is  $t\mapsto \zeta_{k_4} t$.
\end{lemma}

\begin{proof}
Let $d$ be the smallest integer such that $D:=dA^{-1}$ has integral coefficients. 
Let $T=\{\sum X_i^d=0\}\subset \Ps^3$ be the Fermat surface of degree $d$. Then there is a rational map $T\dashrightarrow S$ given by $(X_0:X_1:X_2:X_3)\mapsto ( \prod X_j^{d_{0j}}: \dots: \prod X_j^{d_{3j}})$.  On $T$ there is a natural action of $(\Z/d\Z)^3$, given by  $(X_0:X_1:X_2:X_3)\mapsto (\zeta_d^{a_1}X_0:\zeta_d^{a_2}X_1:\zeta_d^{a_3}X_2:X_3)$. On the affine chart $X_3\neq 0$ with coordinates $x,y,t$ this action is given by $(x,y,t)\mapsto (\zeta_d^{a_1}x,\zeta_d^{a_2}y,\zeta_d^{a_3}t)$.

The rational map $T\dashrightarrow S$ is given (in affine coordinates) by 
\[(x,y,t)\mapsto \left(\frac{x^{d_{00}}y^{d_{01}}t^{d_{02}}}{x^{d_{30}}y^{b_{31}}t^{b_{32}}},\frac{x^{d_{10}}y^{d_{11}}t^{d_{12}}}{x^{d_{30}}y^{d_{31}}t^{d_{32}}},\frac{x^{d_{20}}y^{d_{21}}t^{d_{22}}}{x^{d_{30}}y^{d_{31}}t^{d_{32}}}\right).\]

The action of $(\Z/d\Z)^3$ descents to $S$ and respects the standard fibration. Let $t=X_2/X_3$ be a coordinate on the base of the standard fibration. Then  $(a_1,a_2,a_3)\in (\Z/d\Z)^3$ acts  as $t \mapsto \zeta_d^e t$ with $e\equiv (d_{20}-d_{30})a_1+(d_{21}-d_{31})a_2+(d_{22}-d_{32})a_3\bmod d$.
Since $\bk$ as defined in Remark \ref{remMinFib} is proportional to $(0,0,1,-1)A^{-1}$ it follows that $\bk$ is proportional to $(d_{20}-d_{30},d_{21}-d_{31},d_{22}-d_{32},d_{23}-d_{33})$, i.e., there is an $m\in \Z$ such that $mk_i=d_{2i}-d_{3i}$. In particular, setting $d'=d/m$ it follows that $(a_0,a_1,a_2)$ acts as $t \mapsto \zeta_{d'}^e t$ with $e\equiv k_1a_1+k_2a_2+k_3a_3\bmod d'$.

Let $p$ be a prime number and suppose that $p^m$ divides $k_4$. Since $k_4$ is a divisor of $d$ and the greatest common divisor of the $k_i$ equals one it follows that $p^m$ also divides $d'$. Since the greatest common divisor of the $k_i$ equals one it follows that at least one of the $k_i$ is not divisible by $p$. Without loss of generality we may assume that $k_1$ is invertible modulo $p$. From this it follows that we can choose $a_1$ in such a way that $a_1k_1+a_2k_2+a_3k_3\equiv 1 \bmod p^m$.

The corresponding automorphism $\sigma'_{p^m}$ of $S$ maps $t$ to $t\times \zeta t$ where $\zeta$ is a primitive $p^tn$-th root of unity. Take now $\sigma_{p^m}:=(\sigma'_{p^m})^n$. Then $\sigma_{p^m}$ multiplies $t$ with a primitive $p^t$-root of unity.
Write now $k_4=\prod p_i^{t_i}$. Then $\sigma:=\prod_i \sigma_{p_i^{t_i}}$ multiplies $t$ with a primitive $k_4$-th root of unity.
\end{proof}

\begin{proposition}\label{propBaseChangeOneSingFib} Let $S\to \Ps^1$ be a Delsarte fibration with $\det(A)\neq 0$ then there exists a Delsarte base change $S_n\to \Ps^1$ of $S\to \Ps^1$ which is isomorphic to the base change of a genus $g$ fibration $S_0\to \Ps^1$ with at most one singular fiber outside $0,\infty$.
\end{proposition}

\begin{proof}
 From Lemma~\ref{lemOneT} it follows that we may assume that the Delsarte fibration is  a minimal Delsarte fibration, i.e., we have an affine equation for the generic fiber of the form $m_1+m_2+m_3+tm_4$, where the $m_i$ are monomials in $x$ and $y$. 
 On a minimal Delsarte fibration $\varphi :S\to \Ps^1$ there is an automorphism of order $k_4$ that acts on the $t$-coordinate as $t\mapsto \zeta_{k_4}t$. In particular, all the fixed points of this automorphism are in the fibers over 0 and $\infty$. 

Consider next $\psi:S/\langle\sigma\rangle\to \Ps^1/\langle\sigma\rangle\cong \Ps^1$. Now the singular fibers of $\varphi$ are possibly at $t=0, \infty$ and at $t^{k_4}=\prod k_i^{k_i}$, hence the singular fibers of $\psi$ are possibly at $t=0,\infty$ and $t=\prod k_i^{k_i}$.
\end{proof}

\begin{proposition}\label{propStandardForm} Let $\varphi: S\to \Ps^1$ be a minimal Delsarte fibration with affine equation $m_1+m_2+m_3+tm_4$ such that the general fiber has positive geometric genus. Then one of the following happens
\begin{itemize}
\item $m_4$ equals one of $m_1,m_2,m_3$. In this case the fibration is isotrivial.
\item $S$ is Delsarte-birational to a Delsarte surface with equation of the form $y^a=f(x,t)$.
\item every singular fiber over $t=t_0$ with $t_0\neq 0,\infty$ is semistable.
\end{itemize}
\end{proposition}

\begin{proof} Assume that all four  $m_i$ are distinct. Let $N_i$ be as in Remark~\ref{remMinFib}.
Let $t_0\in K^*$ be such that the fiber over $t=t_0$ is singular. Let $P=(X_0:X_1:X_2)\in \Ps^2$ be a singular point of the fiber. 
From the proof of Proposition~\ref{propDiscrim} it follows that at least one of the $N_i$ is nonzero and that $(N_1:N_2:N_3:N_4)=(k_1:k_2:k_3:\frac{k_4}{t_0})$ holds.
From Remark~\ref{remMinFib} it follows that at most one of the $k_i$ is zero. 

Suppose first that one of the $k_i$, say $k_1$ is zero. This implies that $N_1$ vanishes and that the other $N_i$ are nonzero. Therefore one of the coordinate of $P$ has to be zero (in order to have  $N_1=0$). If two of the coordinates of $P$ are zero then from  $\det(A)\neq0$ it follows that there is some $i\neq 1$ such that $N_i=0$, which contradicts the fact that at most one $N_i$ vanishes. Hence without loss of generality we may assume  $P=(\alpha:0:1)$ with $\alpha\neq 0$, $X_1\mid N_1$ and $X_1\nmid N_i$ for $i=1,2,3$. In particular, we have an affine equation for the fibration of the form $m_1+m_2+m_3+tm_4$, where $y$ divides $m_1$, and $m_2,m_3,m_4$ are of the form $x^{a_i}$.
 Multiply the equation with a power of $x$ such that $m_1$ is of the form $x^{ab}y^{b}$ and set $y_1=y/x^a$. Then we obtain an equation of the form $y_1^b=f(x,t)$, where $f(x,t)$ is of the form $x^a+x^b+tx^c$. This yields the second case.

It remains to consider the case where all the $N_i$ are nonzero. Let $P\in S$ be a point where the fiber over $t=t_0$ singular. Let $f$ be an affine equation for $S$. We prove below that if we localize $K[x,y,z,t]/(f_x,f_y,f_z,t-t_0)$ at $P$ then this ring is isomorphic to $k[x]$. Hence the scheme defined by the Jacobian ideal of fiber at $t=t_0$ has length one at the point $P$. Equivalently, the Milnor number of the singularity of the fiber at $t=t_0$ at the point $P$ equals one. In particular, the singularity of the fiber at $P$ is an ordinary double point.

Consider now the rational map $\tau:\Ps^2\setminus V(X_0X_1X_2) \to \Ps^3$ given by $(X_0:X_1:X_2)\mapsto (N_1:N_2:N_3:N_4)$. The map $\tau$ is unramified at all points $Q\in \Ps^2$ such that  $\tau (Q)\not \in V(X_0X_1X_2X_3)$.

Since we assumed that all the $N_i$ are nonzero it follows that also all the $X_i$ are nonzero. Hence the length of $V(f_{X_0},f_{X_1},f_{X_2},t-t_0)$ at $P$ equals the length of $V(X_0f_{X_0},X_1f_{X_1},X_2f_{X_2},t-t_0)$ at $P$. From the proof of Proposition~\ref{propDiscrim} it follows that $V(X_0f_{X_0},X_1f_{X_1},X_2f_{X_2},t-t_0)$ is the scheme-theoretic intersection of $\ker B_{t_0}$ and $V(\prod z_i^{k_i}-1)$ and that this intersection is locally given by
$V(k_4Z_0-t_0k_1Z_3,k_4Z_1-t_0k_2Z_3,k_4Z_2-t_0k_3Z_4, Z_2-t_0Z_4)$, whence the length of the scheme equals one, and therefore the local Milnor number equals one, and the singularity is an ordinary double point.
\end{proof}

\begin{theorem}\label{thmSingFib} Suppose $S\to \Ps^1$ is a Delsarte fibration of genus $1$ with nonconstant $j$-invariant. Then every singular fiber at $t\neq 0,\infty$ is of type $I_\nu$.
\end{theorem}

\begin{proof}
Without loss of generality we may assume the fibration is a Delsarte minimal fibration. In particular we have an affine equation for this fibration of the form described in the previous Proposition.

In the first case the fibration is isotrivial and therefore the $j$-invariant is constant, hence we may exclude this case. If we are in the third case then each singular fiber at $t=t_0$ is semistable and, in particular, is of type $I_\nu$.

It remains to consider the second case. In this case we have an affine equation of the form $y^a=f(x,t)$. Suppose first that $a>2$ holds. Then the generic fiber has an automorphism of order $a$ with fixed points. This implies that the $j$-invariant of the generic fiber is  either $0$ or $1728$. In particular, the $j$-invariant is constant and that the fibration is isotrivial. Hence we may assume  $a=2$. In this case we have an affine equation $y^2=f(x,t)$. Without loss of generality we may assume  $x^2\nmid f$. Since the generic fiber has genus $1$ it follows that $\deg_x(f)\in \{3,4\}$. Since $S$ is a Delsarte surface it follows that $f$ contains three monomials.

Suppose first that $\deg_x(f)=3$ and that at $t=t_0$ there is a singular fiber of type different from $I_v$. Then $f(x,t_0)$ has a triple root, i.e., $f(x,t_0)=(x-t_0)^3$. This implies that $f(x,t_0)$ consists of either one or four monomials in $x$. This contradicts the fact that $f(x,t)$ consists of three monomials and $t_0\neq 0$.

If $\deg_x(f)=4$ then we may assume (after permuting coordinates, if necessary) that $f=x^4+x^a+t$  or $f=x^4+tx^a+1$.
If the fiber type at $t=t_0$ is different from $I_v$ then $f(x,t_0)$ consists of three monomials and  $y^2=f(x,t_0)$ has at singularity different from a node. In particular, $f(x,t_0)$ has a zero or order at least 3 and therefore $f(x,t_0)=(x-a)^4$ or  $f(x,t_0)=(x-a)(x-b)^3$. In the first case $f(x,t_0)$ contains five monomials, contradicting the fact that is has three monomials. In the second case note that the constant coefficient of $f(x,t_0)$ is nonzero and hence  $ab\neq 0$. Now either  the coefficient of $x$ or of $x^3$ is zero. From this it follows that either $b=-3a$ or $a=-3b$ holds. Substituting this in $f(x,t_0)$ and the the fact that $f(x,t_0)$ has at most three monomials yields $b=0$, contradicting  $ab\neq 0$.
\end{proof}

\begin{corollary}\label{corSingFib} Let $\varphi: S\to \Ps^1$ be an elliptic Delsarte surface, then there exists a cyclic base change of $\varphi$ ramified only at 0 and $\infty$
 that is isomorphic to a cyclic base change, ramified only at $0$ and $\infty$, of an elliptic surface with at most one singular fiber away from 0 and $\infty$ and this fiber is of type $I_v$.
\end{corollary}

Let $\pi:E\to \Ps^1$ be an elliptic surface (with section). Define $\gamma(\pi)$ to be
\[ \gamma(\pi):=\sum_{t\neq0,\infty} \left( f_t-\frac{e_t}{6}\right)-\frac{n_0}{6}-\frac{n_\infty}{6},\]
where $f_t$ is the conductor of $\pi^{-1}(t)$, $e_t$ the Euler number of $\pi^{-1}(t)$ and $n_p$ is zero unless the fiber at $p$ is of type $I_n$ or $I_n^*$ and in this cases $n_p=n$.

In \cite{FastExtremal} and \cite{FastLongTable} Fastenberg studies rational elliptic surfaces with $\gamma<1$. She determines the maximal Mordell-Weil rank of such elliptic surfaces under cyclic base changes of the form $t\mapsto t^n$.

We will now show that each Delsarte fibration of genus 1 with nonconstant $j$-invariant becomes after a Delsarte base change  the base change of a rational elliptic surface with $\gamma<1$. In particular, the maximal Mordell-Weil ranks for Delsarte fibrations of genus 1 under cyclic base change (as presented in \cite[§3.4]{HeijnePhD} and \cite{HeijneMoC}) can also be obtained from \cite{FastLongTable}.
  
\begin{corollary}\label{corGamma} Let $\pi: S\to \Ps^1$ a minimal Delsarte fibration of genus $1$ with nonconstant $j$-invariant. Then $S$ is the base change of a rational elliptic surface with $\gamma<1$.
 \end{corollary}
\begin{proof}
 From Theorem~\ref{thmSingFib} it follows that $\pi$ is the base change of an elliptic fibration $\pi':S'\to \Ps^1$ with at most one singular fiber away from $0$ and $\infty$ and this fiber is of type $I_v$.
 Since the $j$-invariant is nonconstant it follows that $\pi'$ has at least three singular fibers, hence there is precisely one singular fiber away from $0$ and $\infty$.
  Since this fiber is of type $I_\nu$ it follows that   $f_t=1$ for this fiber. Hence
\[ \gamma=1-\frac{e_t+n_0+n_\infty}{6}<1.\]
\end{proof}

\begin{remark} The converse statement to this results holds also true: let $\pi:S\to \Ps^1$ be a rational elliptic surface with $\gamma<1$, only one singular fiber away from $0$ and $\infty$ and this fiber is of type $I_\nu$. Then there exists a base change of the form $t\mapsto t^n$ such that the pullback of $\pi:S\to\Ps^1$ along this base change is birational to the standard fibration on a Delsarte surface. One can obtain this result by comparing the classification of elliptic Delsarte surfaces from \cite[Chapter 3]{HeijnePhD} with the tables in \cite{FastExtremal} and \cite{FastLongTable}.
\end{remark}

\begin{example}
According to \cite{FastLongTable} there is an elliptic surface with a $IV$-fiber at $t=0$, an $I_1$-fiber at $t=\infty$ and one further singular fiber that is of type $I_1^*$, such that the maximal rank under base changes of the form $t\mapsto t^n$ is $9$. Such a fibration has a nonconstant $j$-invariant. Corollary~\ref{corSingFib} now implies that this fibration is not a Delsarte fibration.

If we twist the $I_1^*$ fiber and one of the fibers at $t=0$ or $t=\infty$ then we get the following fiber configurations $IV;I_1^*;I_1$ or $II^*;I_1;I_1$. Then maximal rank under base changes of the form $t\mapsto t^n$ equals 9 in both cases.  Now $y^2=x^3+x^2+t$ has singular fibers of type $I_1$ at $t=0$ and $t=-4/27$ and of type $II^*$ at $t=\infty$ and $y^2=x^3+tx+t^2$ has a $IV$-fiber at $t=0$, a $I_1$ fiber at $t=-4/27$ and a $I_1^*$ fiber at $t=\infty$. Hence both fibration occur as Delsarte fibrations.
\end{example}

\begin{example}
Consider the elliptic Delsarte surface that corresponds to
\[Y^2=X^3+X^2+tX.\]
We can easily compute the discriminant and $j$-invariant of  this fibration:
\[ \Delta=-64t^3-16t^2   \mbox{ and } j=256\frac{(3t-1)^3}{4t^3-t^2}.\]
From this we can see that there are three singular fibres. Over $t=0$ there is $I_2$-fiber, over $t=\infty$ there is a $III$-fiber and over $t=-1/4$ there is a $I_1$-fiber.
We then check that this corresponds to the second entry in the list of  \cite{FastLongTable}.
\end{example}

\begin{remark}\label{rmkDifference} The approaches to determine the maximal Mordell-Weil ranks under cyclic base change in \cite{FastLongTable} and in \cite{HeijnePhD} are quite different. The former relies on studying the local system coming from the elliptic fibration, whereas the latter purely relies on Shioda's algorithm to determine Lefschetz numbers of Delsarte surfaces. This explains why Fastenberg can deal with several base changes where the ``minimal" fibration has four singular fibers (which cannot be covered by Shioda's algorithm because of Proposition~\ref{propBaseChangeOneSingFib}) but cannot deal with fibration with constant $j$-invariant. Instead Shioda's algorithm can handle some of them.
\end{remark}

\section{Isotrivial fibrations}\label{secIsoTrivial}
Using Proposition~\ref{propStandardForm} one easily describes all  possible isotrivial minimal Delsarte fibrations.

\begin{proposition} Suppose the standard fibration on $S$ is isotrivial and that the genus of the generic fiber is positive. Then there is a Delsarte base change and a Delsarte birational map such that the pull back of the standard fibration is of the form $m_1+m_2+(1+t^n)m_3$, $y^3+x^3+x^2+t^n$, or $y^a+x^2+x+t^n$. 
\end{proposition}

\begin{proof} Suppose the affine equation for $S$ is of  the third type of Proposition~\ref{propStandardForm}.  Then $S$ admits a semistable fiber and in particular the fibration cannot be isotrivial. 
If the affine equation for $S$ is of the first type of Proposition~\ref{propStandardForm}, then the generic fiber is (after an extension of the base field) isomorphic to $m_1+m_2+m_3$, in particular each two smooth fibers of the standard fibration are isomorphic and therefore this fibration is isotrivial. In this case $S$ is the pull back of $m_1+m_2+(1+t^n)m_3$.

Hence we may restrict ourselves to the case where we have an affine equation of the form $y^a=x^bf(x,t)$ where $f$ consists of three monomials, $f(0,t)$ is not zero and the  exponent of $x$ in each of the three monomials in $f$ is different.  Moreover, after a Delsarte birational map we may assume that $b<a$.

The surface $S$ is birational to a surface $y^a=x^bz^{c+\deg(f)}f(x/z,t)$ in $\Ps(1,w,1)$, with $0\leq c<a$ and $w=(b+c+\deg(f))/a\in \Z$.
The standard fibration on $S$ is isotrivial if and only if the moduli of the zero-set of  $x^bz^{c+\deg(f)}f(x/z,t)$ in $\Ps^1_{(x:z)}$ is independent of $t$. We will now consider this problem.

We cover first the case where $d':=\deg_x(f)>2$ holds. 

After swapping the role of $x$ and $z$, if necessary, we may assume  that the coefficient of $xz^{d'-1}$ is zero.
We claim that after a map of the form $y=t^{c_1}y, x=t^{c_2}x,z=z,t=t^{c_3}$ we may assume that  $f=x^{d'}+x^{c'}z^{d'-c'}+tz^{d'}$. To see this, take an affine equation for $S$ of the form 
$y^a=x^b(a_1x^{d'}+a_2x^{c'}+a_3)$, where $a_i\in\{1,t\}$, and two of the $a_i$ equal $1$. If $a_1=t$ then we need to take an integer solution of $ac_1=bc_2+d'c_2+c_3=bc_2+c'c_2$, if $a_2=t$ then we need to take an integer solution of $ac_1=bc_2+d'c_2+c_3=bc_2+c'c_2$. In both cases we obtain an affine equation of the form $y^a=x^b(x^{d'}+x^{c'}+t^n)$. This fibration is isotrivial if and only if $y^a=x^b(x^{d'}+x^{c'}+t)$ is isotrivial, which proves the claim.

Hence from now on we assume that $f$ is of the form $x^{d'}+x^{c'}z^{d'-c'}+tz^{d'}$ with $d'>2$ and $c'>1$.

Let $s$ denote the number of distinct zeroes of  $g(x,z):=x^bz^{c+\deg(f)}f(x/z,t)$ for a general $t$-value. 
We say that fiber at $t=t_0$ is bad if  $x^bz^{c+\deg(f)}f(x/z,t)$  has at most $s-1$ distinct zeroes.
The main result from \cite{KT} yields that if the fiber a $t=t_0$ is bad then $g(x,z)$ has at most $3$ distinct zeroes. We are first going to classify all $g$ satisfying this condition. Then we will check case-by-case whether the moduli of the zeroes of $g$ are independent of $t$.

Consider the fiber over $t=0$. From $c'>1$ it follows that  $x=0$ is a  multiple zero of $f(x,0)$. Hence that the fiber over $t=0$ is bad. If $c$ is positive then the criterion from \cite{KT} implies that $g(x,0)$ can have at most one further zero and hence $d'=c'+1$. If $c=0$ then $g$ can have at most two further zeroes and therefore $d'-c'\in \{1,2\}$.

Suppose first  $d'=c'+1$. 
Consider $f'(x,t):=\frac{\partial}{\partial x}f(x,t)$. Our assumption on $f$ implies that $f'(x,t)$ is a polynomial only in $x$. The fiber at $t=t_0$ is bad if and only if $f'(x,t_0)$ and $f(x,t_0)$ have a common zero. From $c'=d'-1$ it follows that $f'(x,t)$ has a  unique zero different from $0$, say $x_0$, and $x_0$ is a simple zero of $f'(x,t)$. Now $f(x_0,t)$ is a linear polynomial in $t$. Hence there is a unique nonzero $t$-value $t_0$ over  which there is a bad fiber. Since $x_0$ is a simple zero of $f'(x,t_0)$ it follows that $(x-x_0)^2$ divides $f(x,t_0)$ and that there are $d'-2$ further distinct zeroes, all different from $0$. Using that $g$ has at most $3$ zeroes it follows that  if both $b$ and $c$ are nonzero then $d'-2=0$, if one of $b,c$ is zero then $d'-2\leq 1$ and if both $b$ and $c$ are zero then $d'-2\leq 2$. Using that we assumed that $d'$ is at least $3$ we obtain the  following possibilities for $g$:
$x^b(x^3+x^2z+tz^3)$, $z^c(x^3+x^2z+tz^3)$, $x^3+x^2z+tz^3$ and $x^4+x^3z+tz^4$.

Suppose now  $c=0$ and $d'=c'+2$. Then $f'$ is of the form $\beta(x^2+\alpha)x^{d'-3}$. In particular, there are two possible $x$-values for a bad point in a bad fiber. If they occur in the same fiber and $b=0$ then $d'\in \{4,5\}$, otherwise $d'\in \{3,4\}$. Since $2\leq c'=d'-2$ we may exclude $d'=3$ and we obtain that the two polynomials $x^4+x^2+t$ and $x^5+x^3+t$ are the only possibilities for $f$. We can exclude $x^5+x^3+t$, since it has bad fibers at $t^2=\frac{-3125}{108}$ and a necessary condition to have $d'=5$ is that there is at most one bad fiber with $t\neq0,\infty$.

If $b>0$ then $d'\leq 4$; in particular we have only $x^b(x^4+x^2+t)$ to check.

Actually only in one of the above cases  the moduli are independent of $t$, namely $g=x^3+x^2z+tz^3$:

Note that  the $j$-invariants of the elliptic curves $y^2=x^3+x^2+t$ and  $y^2=tz^3+z+1$ are not constant, hence the moduli of the zeros of $x^b(x^3+x^2z+tz^3)$ and of $z^c(x^3+x^2z+tz^3)$ depend on $t$ (if $b>0$ resp. $c>0$ holds). Since $x^3+x^2z+tz^3$ has degree 3, the moduli of its zeroes are obviously constant. 

The family of genus one curves $y^2=x^4+x^2+t$ has a semistable fiber at $t=\frac{1}{4}$ and the family of genus one curves  $y^2=x^4+x^3z+tz^4$ has a semistable fiber at $t=\frac{27}{256}$. Hence the moduli of the zeroes of  $x^b(x^4+x^2+t)$ for $b\geq 0$ and of $x^4+x^3z+tz^4$ depend on $t$.

Consider now the final case $d'=2$. Then $f=x^2+x+t$, and therefore automatically two of the three possibilities for $g$, namely $z^c(x^2+xz+tz^2)$ and $x^b(x^2+xz+tz^2)$ have constant moduli since they define 3 points in $\Ps^1$. 
Now $y^{b+2}=z^b(x^2+xz+tz^2)$ and $y^{b+2}=x^b(x^2+xz+tz^2)$ are birationally equivalent up to a Delsarte base change, e.g., take $((x:y:z),t)\mapsto ((z:\frac{y}{t}:\frac{x}{t^{b+2}}),t^{b+2})$. Hence these two cases yield only one case up to isomorphism. We may assume that the affine equation equals $y^b+x^2+x+t^n$.
If $b=c=0$ holds, then the generic fiber is a cyclic cover of $\Ps^1$ ramified at two points, and in particular has genus 0. Hence we can exclude this case.
Finally,  $x^bz^c(x^2+xz+tz^2)$ does not have constant moduli since the $j$-invariant of $y^2=x^3+x^2+tx$ is nonconstant.
\end{proof}

\begin{remark} In the case of $y^a+x^2+x+t$ we may complete the square. This yields a surface that is  isomorphic to $y^a+x^2+1+t$, in particular the fibration is birationally equivalent to a fibration of the first kind. However, they are not Delsarte birational.

In \cite[Section 3.5.1]{HeijnePhD} it is shown that $y^3+x^3+x^2+t$ is birational to $y^2+x^3+t^3+1$, however the given birational map is not a Delsarte birational map.

Hence both exceptional case are fibration that are birational to a fibration of the first type.
\end{remark}

From the previous discussion it follows that almost all minimal isotrivial  Delsarte fibrations  are of the form $m_1+m_2+(1+t)m_3$.
 
We will calculate the Picard numbers for one class of such fibration and consider the behavior of the Picard number under Delsarte base change, i.e., base changes of the form $t\mapsto t^a$.

\begin{example}\label{exaIso}
Let $p=2g+1$ be a prime number.
 Consider the isotrivial fibration $y^2=x^{p}+t^{2ap}+s^{2ap}$ of genus $g$-curves over $\Ps^1_{(s:t)}$.   This equation defines a quasi-smooth surface $S$ of degree $2ap$ in $ \Ps(2a,ap,1,1)$. The surface $S$ has one singular point, namely at $(1:1:0:0)$. A single blow-up of this suffices to obtain a smooth surface $\tilde{S}$. The Lefschetz number of $\tilde{S}$ can be computed by using Shioda's algorithm, which we do below. The exceptional divisor of $\tilde{S}\to S$ is a smooth rational curve. In particular, using the Mayer-Vietoris sequence one easily obtains that $h^2(\tilde{S})=h^2(S)+1$ and $\rho(\tilde{S})=\rho(S)+1$.
 Since $S$ is quasi-smooth one has a pure weight 2 Hodge structure on $H^2(S)$. To determine the Hodge numbers of this Hodge structure we use a method of Griffiths and Steenbrink. Note first that $\dim H^2(S)_{\prim}=h^2(S)-1=h^2(\tilde{S})-2$. 
 
   Let $R$ be the Jacobian ring of $S$, i.e.,
   \[ R=\C[x,y,t,s]/\left(\frac{\partial f}{\partial x},\frac{\partial f}{\partial y},\frac{\partial f}{\partial s},\frac{\partial f}{\partial t}\right)=\C[x,y,s,t]/(x^{p-1},y,t^{2p-1},s^{2p-1})\]
 This is a graded ring with weights $(2a,ap,1,1)$. Let $d=2ap$ be the degree of $S$ and $w=ap+2a+2$ the sum of the weights.

 From Griffiths-Steenbrink \cite{SteQua} it follows that 
$H^{2-q,q}(S)_{\prim}$ is isomorphic with
\begin{eqnarray*} R_{qd-w}&=&\spa \{ x^it^js^k\mid 2ai+j+k=qd-w, 0\leq i<p-1, 0\leq j,k< 2p-1\}\\ &=& \spa\{ yx^it^js^k \mid  2ai+j+k=qd, 0< i\leq p-1, 0<j,k\leq 2p-1\}. \end{eqnarray*}
In other words the basis elements of $R_{qd-w}$ correspond one-to-one to vectors
\[\left\{ \left( \frac{1}{2},\frac{i}{p},\frac{j}{2ap},\frac{k}{2ap}\right) \in (\Q/\Z)^4 \left| \begin{array}{c} i,j,k\in \Z;0<i<p; 0<j,k<2ap; \\  \frac{1}{2}+\frac{i}{p}+\frac{j}{2ap}+\frac{k}{2ap}=q\end{array} \right.  \right\}.\]
 
In \cite[Section 2.1]{HeijnePhD} a variant of Shioda's algorithm \cite{ShiodaPic}  is presented. This algorithm calculates the Lefschetz number of a resolution of singularities of a Delsarte surface in $\Ps^3$. In our case we apply this algorithm to the surface $T\subset \Ps^4$ given by
\[-Y^2Z^{2ap-2}+X^pZ^{2ap-p}+W^{2ap}+Z^{2ap}\]
Since the Lefschetz number is a birational invariant, one has that the Lefschetz number of $\tilde{S}$ and $\tilde{T}$ coincide.

Following \cite{HeijnePhD} we need to take the exponent matrix
\[A=
\left(
\begin{array}{cccc}
0&2&0&2ap-2\\
p&0&0&(2a-1)p\\
0&0&2ap&0\\
0&0&0&2ap\\
\end{array}
\right)
\]
and then to determine the three vectors \[\bv_1:=A^{-1}(1,0,0,-1)^T, \bv_2:=A^{-1}(1,0,0,-1)^T\mbox{ and }\bv_3:=A^{-1}(0,0,1,-1)^T.\]
In our case this yields the vectors \[\bv_1=\left(0,\frac{1}{p},0,\frac{-1}{p}\right) ,  \bv_2=\left(\frac{1}{2},0,0,\frac{-1}{2}\right)\mbox { and }\bv_3=\left(0,0,\frac{1}{2ap},\frac{-1}{2ap}\right).\]
 
Consider now the set $L:=i\bv_1+k\bv_2+j\bv_3\in \Q/\Z$. This are precisely the vectors of the form
\[\left\{\left( \frac{k}{2},\frac{i}{p},\frac{j}{2ap},\frac{-apk-2ai-j}{2ap} \right)\in (\Q/\Z)^4\left| i,j,k\in \Z\right.\right\} \]

Let $L_0\subset L$ be the set of vectors $v\in L$ such that none of the entries of $v$ equals $0$ modulo $\Z$, i.e.,
\[ \left\{\left( \frac{1}{2},\frac{i}{p},\frac{j}{2ap},\frac{-ap-2ai-j}{2ap} \right)\in (\Q/\Z)^4\left| \begin{array}{l} i,j\in \Z, 0<i<p,0<j<2ap, \\ j \not \equiv -ap-2ai \bmod 2ap \end{array}\right.\right\} \]
Note that $\#L_0$ is precisely $h^2(S)_{\prim}$.

For an element $\alpha\in \Q/\Z$ denote with $\fr{\alpha}$ the fractional part, i.e., the unique element $\beta\in \Q\cap [0,1)$ such that $\alpha-\beta\equiv 0 \bmod \Z$ and with $\ord_{+}(\alpha)$ the smallest integer $k>0$ such that $k\alpha\in \Z$.

Define the following $\Lambda\subset L_0$  consisting of  elements $(\alpha_1,\alpha_2,\alpha_3,\alpha_4)\in L_0$ such that there is a $t\in \Z$ for which $\ord_+(\alpha_kt)=\ord_+(\alpha_k)$ holds for $k=1,2,3,4$ and  $\fr{t\alpha_1}+\fr{t\alpha_2}+\fr{t\alpha_3}+\fr{t\alpha_4}\neq 2$. The condition $\ord_+(\alpha_kt)=\ord_+(\alpha_k)$ for $k=1,2,3,4$ is equivalent with $t$ being invertible modulo $2a'p$, where $a'=a/\gcd(a,j)$. 
Then the Lefschetz number $\lambda=h^2(\tilde{T})-\rho(\tilde{T})$ equals $\#\Lambda$. 

Since $\lambda(\tilde{S})=\lambda(\tilde{T})$ and $h^2(\tilde{S})=2+\# L_0$ it follows that $\rho(S)$ equals
\[2+\# \left\{ (\alpha_1,\alpha_2,\alpha_3,\alpha_4)\in L_0 \mid\begin{array}{l} \frd{t\alpha_1}+\frd{t\alpha_2}+\frd{t\alpha_3}+\frd{t\alpha_4}=2 \mbox{ for } t\in \Z \\\mbox{ such that }  \ord_{+}(t\alpha_k)=\ord_{+}(\alpha_k), k=1,2,3,4\end{array}\right \}.\]
We now determine this set.

Consider now a vector $\bv$ from $L_0$, i.e., a vector
\[ \left(\frac{1}{2},\frac{i}{p},\frac{j}{2ap},\frac{ap-2ai-j}{2ap}\right) \]
with $i,j\in \Z$, $i\nequiv 0 \bmod p, j\nequiv 0 \bmod 2ap,  ap-2api-j \nequiv 0 \bmod 2ap$.

Take $t\in \{1,\dots, 2a'p-1\}$ such that $\gcd(t,2a'p)= 1$ and $t\equiv i^{-1} \bmod p$. Then $v\in \Lambda$ if and only if $t\bv \in \Lambda$. Hence to determine whether a vector is in $\Lambda$ it suffices to assume $i\equiv 1 \bmod p$.

Suppose now that $p>7$. 
In Lemma~\ref{lemExcl} we show that $\bv\not \in \Lambda$ if and only if the fractional part $\fr{\frac{j}{2ap}}$ is in the set $ \left\{ \frac{p-1}{2p}, \frac{1}{2}, \frac{p+2}{2p},\frac{2p-4}{2p}, \frac{2p-2}{2p}, \frac{2p-1}{2p}\right\}$.  Each of the six values for $j$ yields $(p-1)$ elements in $L_0\setminus \Lambda$, hence $\rho(\tilde{S})=2+6(p-1)$.

One can easily find several divisors on $\tilde{S}$. We remarked in the introduction that the pull back of the hyperplane class and the exception divisor yield two independent classes in $NS(\tilde{S})$. We give now $2(p-1)$ further independent classes:
Let $\zeta$ be $p$-th root of unity. Let $C_{1,i}$ be $x=t^{2a}\zeta^i$, $y=s^{ap}$, $C_{2,i}$ be $x=s^{2a}\zeta^i$, $y=t^{ap}$. Then $\sum_i C_{1,i}$ and $\sum C_{2,i}$ equal the hyperplane class in $Pic(S)$. However each $C_{i,j}$ is nonzero in $\Pic(S)$: Let $\sigma$ be the automorphism of $S$ sending $t$ to $t$ times a $2ap$-th root of unity and leaving the other coordinates invariant. Then the characteristic polynomial of $\sigma$ on the image of $C_{1,i}$, $i=1,\dots,p-1$ is $(t^p-1)/(t-1)$. In particular the image of $\spa \{[C_{1,j}]\}$ is either 0 or $p-1$ dimensional. (One can exclude the former possibility by checking intersection numbers.)
 Similarly using the automorphism $s$ is mapped to $s$ times a $2ap$-th root of unity it follows that $\spa\{[C_{2,i}]\}$ has dimension $p-1$ and that $\spa\{ [C_{i,j}]\}$ is $2(p-1)$ dimensional.
\end{example}

\begin{lemma}\label{lemExcl} Suppose $p>7$.
Let
\[\bv=\left( \frac{1}{2},\frac{1}{p},\frac{j}{2ap},\frac{-2a-ap-j}{2ap} \right)\in (\Q/\Z)^4\]
 such that $j\nequiv 0 \bmod 2ap, 2a+j+ap\nequiv 0 \bmod 2ap$.
Then $\bv\not \in\Lambda$ if and only if 
\[ \frac{j}{2ap} \in\left\{\frac{p-1}{2p}, \frac{1}{2},\frac{p+2}{2p}, \frac{2p-4}{2p},\frac{2p-2}{2p},\frac{2p-1}{2p}\right\}.\]
\end{lemma}

\begin{proof}
Without loss of generality we may assume that $\gcd(a,j)=1$.

We start by proving that if a prime $\ell\geq 5$ divides $a$ then $v \in \Lambda$. For this it suffices to give a $t$, invertible modulo $2ap$ such that
\[\frd{\frac{t}{2}}+\frd{\frac{t}{p}}+\frd{\frac{tj}{2ap}}+\frd{\frac{(-2a-ap-j)t}{2ap}}=1.\]
Since the left hand side is an integer for any choice of $t$, and each summand is smaller than one it suffices to prove that
\[\frd{\frac{t}{2}}+\frd{\frac{t}{p}}+\frd{\frac{tj}{2ap}} \leq 1.\]

Consider the value 
\[t=1+ck\frac{2ap}{\ell},\]
with $c\equiv j^{-1}\bmod \ell$ and $k\in \Z$ such that $k\not\equiv (c\frac{2ap}{\ell})^{-1} \bmod \ell$ and $k$ in the interval 
\[\left(-\frac{\ell j}{2ap},-\frac{\ell j}{2ap}+\frac{\ell(p-2)}{2p}\right).\]
Note that we have to assume $p>7$ or $\ell\geq 5$ to ensure the existence of such a $k$.

Then $\fr{\frac{t}{2}}=\frac{1}{2}$ and 
\[ \frd{\frac{t}{p}} = \frd{\frac{1}{p}+ck \frac{2a}{\ell}} =\frac{1}{p}.\]
Moreover, we have that
\[\frd{\frac{tj}{2ap}}= \frd{\frac{(1+ck\frac{2ap}{\ell})j}{2ap}}=\frd{\frac{j}{2ap}+\frac{k}{\ell}}\leq \frac{(p-2)}{2p}.\]
From this it  follows that
\[\frd{\frac{t}{2}}+\frd{\frac{t}{p}}+\frd{\frac{tj}{2ap}}\leq 1\]
holds, which finishes this case.

Suppose now that the only primes dividing $a$ are $2$ or $3$.  If $p=11,13,17$ and $a=3$ then one can find easily by hand a $t$-value such that
\[\frd{\frac{t}{2}}+\frd{\frac{t}{p}}+\frd{\frac{tj}{2ap}}+\frd{\frac{(-2a-ap-j)t}{2ap}}=1\]
holds. For all other combinations $(a,j,p)$ with $a>1$ we give a  value for $t$ in Table~\ref{TabTVal} such that the above formula holds.

\begin{table}[hbtp]\label{TabTVal}
\begin{tabular}{|c|c|c|c|}
\hline
$a$  &  $\frac{j}{2ap} \in I $&   & $t$ \\
\hline
$4\mid a$& $(0,\frac{p-2}{2p})$ &&$1$\\
$p>3$ & $ (\frac{1}{2},\frac{p-1}{p})$& & $1+ap$\\
      & $ (0,\frac{p-4}{4p})\cup(\frac{3}{4},1)$& $j\equiv 1 \bmod 4$ & $1+\frac{ap}{2}$\\
      & $ (\frac{1}{4},\frac{3p-4}{4p})$ & $j\equiv 1 \bmod 4$ & $1+\frac{3ap}{2}$\\   
      & $(0,\frac{p-4}{4p})\cup(\frac{3}{4},1)$ & $j\equiv 3 \bmod 4$ & $1+\frac{3ap}{2}$\\
    & $(\frac{1}{4},\frac{3p-4}{4p})$ & $j\equiv 3 \bmod 4$ & $1+\frac{ap}{2}$\\
 \hline
$2\mid a$, $4\nmid a'$ &  $(0,\frac{p-2}{2p})$ &&$1$\\
$p>7$              & $  (\frac{1}{2},\frac{p-1}{p})$ && $1+ap$\\
                   & $(0,\frac18-\frac1p)\cup(\frac{3}{8},\frac{5}{8}-\frac{1}{p})\cup (\frac78,1)$ & $j\equiv 1 \bmod 4$ & $2+\frac{ap}{2}$\\
                   & $ (\frac{3}{8},\frac{5}{8}-\frac{1}{p})$&$j\equiv 3 \bmod 4$ & $2+\frac{3ap}{2}$\\

\hline
$9\mid a$& $ (0,\frac{p-2}{2p})$ &&$1$\\
$p>5$     &  $(\frac{1}{3},\frac{5}{6}-\frac{1}{p})$ & $j\equiv 2 \bmod 3$ & $1+\frac{2ap}{3}$\\
      &  $(0, \frac{1}{3}-\frac{1}{p})\cup (\frac{2}{3},1)$ &  $j\equiv 2 \bmod 3$ & $1+\frac{4ap}{3}$\\
      &  $(\frac{1}{3},\frac{5}{6}-\frac{1}{p})$ & $j\equiv 1 \bmod 3$ & $1+\frac{4ap}{3}$\\
      & $(0, \frac{1}{3}-\frac{1}{p})\cup (\frac{2}{3},1)$ &  $j\equiv 1 \bmod 3$ &$1+\frac{2ap}{3}$\\ 

\hline

$a=3$ & $  (0,\frac{p-2}{2p})$ &&$1$\\
$p\equiv 1\bmod 3$  &  $ (\frac{1}{3},\frac{5}{6}-\frac{1}{p})$ &$j\equiv 1 \bmod 3$   & $1+4p$ \\
$p>18$             &  $ (\frac{8}{9},\frac{19}{18}-\frac{1}{p})$&  $j\equiv 1 \bmod 3$ &$3+2p$ \\
                   &  $ (\frac{7}{9},\frac{17}{18}-\frac{1}{p})$ &$j\equiv 1 \bmod 3$ &$3+4p$ \\

                   &  $(\frac{2}{3},1)$ & $j\equiv 2 \bmod 3$   & $1+4p$ \\
                   &  $ (\frac{4}{9},\frac{11}{18}-\frac{1}{p})$ &  $j\equiv 2 \bmod 3$ &$3+2p$ \\
                   &  $ (\frac{5}{9},\frac{13}{18}-\frac{1}{p})$ & $j\equiv 2 \bmod 3$ &$3+4p$ \\
\hline
$a=3$ &           $  (0,\frac{p-2}{2p})$ &&$1$\\
$p\equiv 2\bmod 3$  &  $(\frac{2}{3},1)$ & $j\equiv 1 \bmod 3$   & $1+2p$ \\
$p>18$             &  $(\frac{5}{9},\frac{13}{18}-\frac{1}{p})$ &  $j\equiv 1 \bmod 3$ &$3+2p$ \\
                   &  $  (\frac{4}{9},\frac{11}{18}-\frac{1}{p})$ &  $j\equiv 1 \bmod 3$ &$3+4p$ \\

                   &  $ (\frac{1}{3},\frac{5}{6}-\frac{1}{p})$   & $j\equiv 2 \bmod 3$  &$1+2p$ \\
                   &  $ (\frac{7}{9},\frac{17}{18}-\frac{1}{p})$ &  $j\equiv 2 \bmod 3$ &$3+2p$ \\
                   &  $ (\frac{8}{9},\frac{19}{18}-\frac{1}{p})$ &  $j\equiv 2 \bmod 3$ &$3+4p$ \\
\hline
\end{tabular}
\caption{$t$-values for the case $a=2^{v_2}3^{v_3}$, $a\neq 1$}
\end{table}

The only case left to consider is the case $a=1$. If $p\leq 30$ then one can easily find an appropriate $t$-value by hand. Hence we
may assume that $p>30$. 
If we take $t=1$ then we see that $v\in\Lambda$ whenever 
\[\frac{j}{2ap}=\frac{j}{2p}\in \left(0,\frac{1}{2}-\frac{1}{p}\right).\]
We will consider what happens if $\frac{j}{2p}>\frac{p-2}{2p}$.

Suppose $t<p$ is an odd integer, and $k$ is an integer such that $k \leq \frac{tj}{2p}< k+1$. Then we have
\[ \frd{\frac{t}{2}}+\frd{\frac{t}{p }}+\frd{\frac{tj}{2p}} = \frac{1}{2}+\frac{t}{p}+\frac{tj}{2p}-k\]
The right hand side is at most $1$ if
\[ \frac{j}{2p} \leq \frac{1+2k}{2t}-\frac{1}{p}\]
Hence if
\[ \frac{j}{2p} \in \left( \frac{k}{t},\frac{1+2k}{2t}-\frac{1}{p} \right)\]
Then $\bv\in \Lambda$.

If we take $k=t-1$ then we get the interval
\[ I_t:=\left( 1-\frac{1}{t},1-\frac{1}{2t}-\frac{1}{p}\right)\]
and if we take $k=(t+1)/2$ then we get
\[ I'_t:=\left(\frac{1}{2}+\frac{1}{2t},\frac{1}{2}+\frac{1}{t}-\frac{1}{p}\right).\]
Note that $I_3=I'_3$. 

We claim that if $p>30$ and $5\leq t\leq  \frac{p-1}{2}-3$ then $I'_t\cap I'_{t-2}\neq \emptyset$ and $I_t\cap I_{t-2}\neq \emptyset$.
For this it suffices to check that 
\[ \frac{1}{2}+\frac{1}{2(t-2)} < \frac{1}{2}+\frac{1}{t}-\frac{1}{p} \mbox{ and }1-\frac{1}{2(t-2)}-\frac{1}{p} >1-\frac{1}{t} \]
Both conditions are equivalent with
\begin{equation}\label{eqnBoundPnt} 2t^2-(p+4)t+4p<0\end{equation}
The smallest value to check is $t=5$ then the above formula yields that $p>30$, which is actually the case. For fixed $p$ we have that the above bound is equivalent with $t\in (\frac{1}{4}p+1-\frac{1}{4} \sqrt{p^2-24p+16}, \frac{1}{4}p+1+\frac{1}{4} \sqrt{p^2-24p+16})$. 
The previous argument already shows that the left boundary of this interval is smaller than $5$. Substituting $t=\frac{p-1}{2}-3$ in (\ref{eqnBoundPnt}), yields that for $p>77/3$ the boundary point on the right is at bigger than $\frac{p-1}{2}-3$. In particular, if $p>30$, $t$ is odd then  $I'_t\cap I'_{t-2}\neq \emptyset$ and $I_t\cap I_{t-2}\neq \emptyset$.
Take now the union of $I'_t$ and $I_t$ for all odd $t$ between $3$ and $\frac{p-1}{2}-5$. This yields an interval $I=(\alpha,\beta)$ such that for all $\frac{j}{2p}\in I$ we have that $\bv \in \Lambda$. 
The maximal $t$-value is either  $\frac{p-1}{2}-3$ or $\frac{p-1}{2}-4$ (depending on $p\bmod 4$). Hence we know only that the maximal $t$ is at least $\frac{p-1}{2}-4$. From this it follows that
\[I\supset  \left( \frac{1}{2}+\frac{1}{p-9} , 1-\frac{1}{p}-\frac{1}{p-9}\right).\]
Note that $p-9>\frac{2}{3}p$ and hence $\frac{1}{p}+\frac{1}{p-9}\leq \frac{5}{2p}$. Hence the only possibilities for $\frac{j}{2p} \not \in I$ and $p-2<j <2p$ are
\[  \left\{\frac{p-1}{2p} ,\frac{p}{2p},\frac{p+1}{2p},\frac{p+2}{2p},  \frac{2p-4}{2p}. \frac{2p-3}{2p},\frac{2p-2}{2p},\frac{2p-1}{2p}\right\}.\]
If $\frac{j}{2p}\in \{\frac{p+1}{2p},\frac{2p-3}{2p}\}$ then we have that $\bv$ is in $\Lambda$. This can be verified by taking $t=p-2$. Hence we have shown that for all but six values for $\frac{j}{2p}$ the corresponding vector is in $\Lambda$.

It remains to show that for the remaining  values of $\frac{j}{2p}$ we have that $\bv \not\in \Lambda$.
If $\frac{j}{2p}\in \{\frac{1}{2},\frac{2p-2}{2p}\}$ then two  coordinates $\alpha,\beta$ of $\bv$ equal $\frac{1}{2}$. Hence for any admissible $t$ we have 
\[\frd{\frac{t}{2}}+\frd{\frac{t}{p}}+\frd{\frac{tj}{2ap}}+\frd{\frac{(-2a-ap-j)t}{2ap}}> \frac{1}{2}+\frac{1}{2} =1\]
Since the left hand side is an integer, it is at least 2.

In the other four cases we have two entries $\alpha,\beta$ such $\alpha=\beta+\frac{1}{2}$. Since $t$ is odd we have then that $|\fr{t\alpha}-\fr{t\beta}|=\frac{1}{2}$ and therefore
\[\frd{\frac{t}{2}}+\frd{\frac{t}{p}}+\frd{\frac{tj}{2p}}+\frd{\frac{(-2-p-j)t}{2p}}>\frd{\frac{1}{2}} +\frd{t\alpha}+\frd{t\beta}>1.\]
Summarizing  we have that for all $t$ that are invertible modulo $2p$ that 
\[\frd{\frac{t}{2}}+\frd{\frac{t}{p}}+\frd{\frac{tj}{2p}}+\frd{\frac{(-2-p-j)t}{2p}}\geq 2\]
holds. Using the symmetry of the coordinates it follows that  for all $t$ that are invertible modulo $2p$ we have
\[\frd{\frac{t}{2}}+\frd{\frac{t}{p}}+\frd{\frac{tj}{2p}}+\frd{\frac{(-2-p-j)t}{2p}}\leq 2\]
hence $\bv \not \in \Lambda$, which finishes the proof.
\end{proof}

\begin{remark}
If we had taking $p$ to be a non-prime the result would be different.
This would restrict the number of possible $t$'s that can be used.
As a consequence the Picard number will probably be slightly higher.

The cases for $p=7$, $p=5$ and $p=3$ can also be computed.
If $p=7$ and $3|a$ we  get $\rho(\tilde{S})=2+14(p-1)$.
For $p=5$ and $6|a$ we  get $\rho(\tilde{S})=2+18(p-1)$.
For $p=3$ and $60|a$ we get $\rho(\tilde{S})=2+30(p-1)$.
Also for small $p$ the result would be higher.
\end{remark}

\bibliographystyle{plain}
\bibliography{fastlist}

\end{document}